\documentclass[runningheads,a4paper]{llncs}

\usepackage{
amsmath,
amscd,
amssymb,
hyperref
}
\usepackage{comment}
\usepackage{tikz}
\usetikzlibrary{positioning,arrows}

\usepackage{graphicx}

\usepackage{url}
\urldef{\mails}\path|{jkari, vosalo, iatorm}@utu.fi|
\newcommand{\keywords}[1]{\par\addvspace\baselineskip
\noindent\keywordname\enspace\ignorespaces#1}

\newcommand{\Z}{\mathbb{Z}}

\newcommand{\N}{\mathbb{N}}
\newcommand{\M}{\mathbb{M}}
\newcommand{\B}{\mathcal{B}}
\newcommand{\X}{\mathcal{X}}

\newcommand{\emptycell}{
	\begin{tikzpicture}
	\draw[thick,white] (.25,0) -- (0,.25);
	\draw (0,0) rectangle (.25,.25);
	\end{tikzpicture}
}
\newcommand{\leftcell}{
	\begin{tikzpicture}
	\draw (0,0) rectangle (.25,.25);
	\draw[thick] (.25,0) -- (0,.25);
	\end{tikzpicture}
}
\newcommand{\rightcell}{
	\begin{tikzpicture}
	\draw (0,0) rectangle (.25,.25);
	\draw[thick,] (0,0) -- (.25,.25);
	\end{tikzpicture}
}
\newcommand{\bothcell}{
	\begin{tikzpicture}
	\draw (0,0) rectangle (.25,.25);
	\draw[thick] (.25,0) -- (0,.25);
	\draw[thick] (0,0) -- (.25,.25);
	\end{tikzpicture}
}

\begin{document}

\mainmatter

\title{Trace Complexity of Chaotic Reversible Cellular Automata\thanks{Research supported by the Academy of Finland Grant 131558}}

\author{
Jarkko Kari
\and
Ville Salo
\and
Ilkka T\"orm\"a
}

\institute{
		TUCS -- Turku Centre for Computer Science, Finland, \\
		University of Turku, Finland \\
		\mails
}

\maketitle

\begin{abstract}
Delvenne, K\r{u}rka and Blondel have defined new notions of computational complexity for arbitrary symbolic systems, and shown examples of effective systems that are computationally universal in this sense. The notion is defined in terms of the trace function of the system, and aims to capture its dynamics. We present a Devaney-chaotic reversible cellular automaton that is universal in their sense, answering a question that they explicitly left open. We also discuss some implications and limitations of the construction.
\end{abstract}

\section{Introduction}
\label{sec:Intro}

A significant branch of dynamical systems research is the study of computability and computational complexity of finitely presented systems. In the literature, there are usually multiple incomparable notions of computability and computational universality for a sufficiently popular model, like cellular automata \cite{Sm71,Wo84,Co04}. Traditionally, for a model to be considered computationally universal, it is sufficient for it to be able to simulate the computation process of any Turing machine in a suitably transparent way. However, seemingly minor variations to the formal definition (if one is presented) may reduce a universal system into a trivial one. A related notion is that of \emph{intrinsic} universality, which refers to the ability of simulating any other instance of the same model in some formally defined way. In cellular automata, intrinsic universality is usually defined with respect to block simulations, although in earlier research this notion had usually also been left undefined. See \cite{DuRo99} for a discussion on the implications of not defining these notions rigorously. In the context of reversible computation, intrinsic universality of reversible Turing machines has been discussed in \cite{AxGl11}.

In \cite{DeKuBl06}, a new definition of computational universality was proposed that can be applied to a wide range of discrete dynamical systems, including cellular automata, shift spaces, tag systems, and Turing machines, which can be viewed as dynamical systems in more than one way \cite{Ku97}. It is an update of the definition given in \cite{DeKuBl05}, and aims to capture the dynamical complexity of the system, so that systems that are dynamically too simple (like the identity map on a set) or allow too much freedom (like the shift map) would not be universal.

The computational universality presented in \cite{DeKuBl06} intuitively means the hardness of deciding prediction problems like `for subsets $U, V, W$ of the state space, is there a point in $U$ that is mapped by the dynamics to $V$, and stays there until it enters $W$.' For example, one would show that Turing machines are computationally universal by defining $U$ as the singleton set containing the initial configuration, $V$ as the set of all configurations, and $W$ as the set of final configurations. Of course, for the definition to be sensible, the subsets need to be restricted in some way. In symbolic systems, whose elements are infinite sequences of symbols, we require that the sets are clopen, that is, they are defined by the contents of finitely many coordinates.

One of the main observations in \cite{DeKuBl06} is that universal systems tend to be `at the edge of chaos': the dynamics appears chaotic, but has underlying structure that gives rise to the universality. They give examples of effective systems that are both universal and chaotic in the sense of Devaney \cite{De89}, but the existence of a universal chaotic cellular automaton is explicitly left open. In this article, we contruct a reversible universal chaotic cellular automaton, answering this question in the positive. Note that although reversible cellular automata were shown to be able to simulate arbitrary computation already in \cite{MoHa89}, and their construction seems to be universal in the sense of \cite{DeKuBl06}, it is not chaotic.

\section{Definitions}

Let $\M$ be either $\N$ or $\Z$, and let $S$ be a finite alphabet. The set $S^\M$, equipped with the product topology, is called the \emph{full $\M$-shift on $S$}, and its elements are called \emph{configurations}. The monoid $(\M, {+})$ acts on $S^\M$ by the \emph{shift maps} $\sigma^m : S^\M \to S^\M$ for $m \in \M$, defined by $\sigma^m(x)_n = x_{n + m}$. We denote $\sigma^1 = \sigma$. For a word $w \in S^n$ and $x \in S^\M$, we say that $w$ \emph{occurs in $x$}, denoted $w \sqsubset x$, if there exists $m \in \M$ such that $w = x_{[m, m+n-1]}$. This notation is extended to sets of configurations in the obvious way. An \emph{$\M$-shift space} is a topologically closed set $X \subset S^\M$ satisfying $\sigma(X) \subset X$. Equivalently, a shift space is defined by a set $F \subset S^*$ of \emph{forbidden words} as $\X_F = \{ x \in S^\M \;|\; \forall w \in F : w \not\sqsubset x \}$. If $F$ is finite, $\X_F$ is a \emph{shift of finite type} (SFT for short). We denote $\B_n(X) = \{ w \in S^n \;|\; w \sqsubset X \}$ and $\B(X) = \bigcup_{n \in \N} \B_n(X)$.

A \emph{block map} is a continuous function $f : X \to Y$ between shift spaces $X, Y \subset S^\M$ that satisfies $f \circ \sigma|_X = \sigma|_Y \circ f$. Alternatively, a block map is defined by a \emph{local rule} $\hat f : \B_{a+m+1}(X) \to \B_1(Y)$, where $a, m \in \N$ are the \emph{anticipation} and \emph{memory} of $\hat f$, by $f(x)_n = \hat f(x_{[n-m, n+a]})$. The interval $\{-m, \ldots, a\} \subset \M$ is called the \emph{neighborhood} of $\hat f$. In the case $\M = \N$, we must have $m = 0$. If $X = Y = S^\M$, then $f$ is called a \emph{cellular automaton} (CA for short), and a bijective CA is called \emph{reversible}, since its inverse function is also a CA. We sometimes identify a CA and its local rule, but this should always be clear from the context.

A \emph{symbolic system} is a tuple $(X, f)$, where $X$ is a compact metric space with countable clopen basis (equivalently, homeomorphic to a closed subset of a full shift), and $f : X \to X$ is continuous. The system is \emph{effective} if the clopen basis of $X$ can be enumerated so that complementation, intersection and $f$-preimage are computable operations. It is \emph{chaotic} (in the sense of Devaney \cite{De89}) if
\begin{itemize}
\item it is sensitive (there exists $\epsilon > 0$ such that for all $x \in X$ and $\delta > 0$ there exist $y \in X$ and $n \in \N$ with $d(x,y) < \delta$ and $d(f^n(x),f^n(y)) \geq \epsilon$),
\item it is transitive (for all nonempty open sets $U, V \subset X$, there exists $n \in \N$ with $U \cap f^n(V) \neq \emptyset$), and
\item the $f$-periodic points (those $x \in X$ for which $f^n(x) = x$ for some $n \in \N$) are dense in $X$.
\end{itemize}
In particular, every $\Z$-shift space $(X, \sigma)$ with the left shift is a symbolic system, as is $(X, f)$ for every block map $f : X \to X$. These are the only kinds of symbolic systems we use in this article; the full definition is given only for completeness, and to state the definitions of universality given in \cite{DeKuBl06}.

\begin{example}
For a cellular automaton $f : S^\Z \to S^\Z$, most dynamical notions have combinatorial characterizations. For example, $f$ is transitive if and only if for all words $u, v \in S^{2 \ell + 1}$ of the same odd length, there exists a configuration $x \in S^\Z$ and $n \in \N$ such that $x_{[-\ell, \ell]} = u$ and $f^n(x)_{[-\ell, \ell]} = v$.
\end{example}

A \emph{Muller automaton} is a quintuple $A = (Q, q_0, \Sigma, \delta, F)$, where $Q$ is a finite \emph{state set}, $q_0 \in Q$ an \emph{initial state}, $\Sigma$ a finite \emph{input alphabet}, $\delta : Q \times \Sigma \to Q$ a \emph{transition function} and $F \subset 2^Q$ a set of \emph{accepting subsets of states}. A Muller automaton runs deterministically on infinite words $w \in \Sigma^\N$ analogously to a standard finite automaton, and accepts if the set of states that are visited infinitely often during the computation is in the set $F$. The language accepted by $A$ is denoted $L_A$. For a language $L \subset A^*$, we denote by $L^\omega \subset A^\N$ the set of infinite concatenations of the words of $L$. In particular, if $L$ is regular, then $L^\omega$ is accepted by a Muller automaton. See \cite{Th90} for a reference on Muller automata, and other types of finite automata on infinite words.

In this article, a \emph{Turing machine} is a sextuple $M = (Q, \Sigma, q_0, q_f, B, \delta)$, where $Q$ is a finite \emph{state set}, $\Sigma$ a finite \emph{input alphabet}, $q_0, q_f \in Q$ are the \emph{initial and final states}, $B \in \Sigma$ is the \emph{blank letter} and $\delta \subset Q \times (\Sigma \times \Sigma \cup \{{/}\} \times \{+,0,-\}) \times Q$ a \emph{transition relation}. Turing machines are run on two-way infinite tapes, and the initial input is placed immediately to the right of the head. The interpretation of a quadruple $[q_1, a, b, q_2] \in \delta$ in the case $a, b \in \Sigma$ is that if $M$ is in state $q_1$ and reading the letter $a$, it may rewrite it to $b$ and go to state $q_2$. In the case $a = {/}$ and $b \in \{+,0,-\}$, if $M$ is in state $q_1$, it may go to state $q_2$ and move one step in the direction indicated by $b$. Two quadruples $[q_1, a, b, q_2]$ and $[q'_1, a', b', q'_2]$ \emph{overlap in domain} if $q_1 = q'_1$ and $a, a' \in \Sigma \Longrightarrow a = a'$. They \emph{overlap in range} if $q_2 = q'_2$ and $a, a' \in \Sigma \Longrightarrow b = b'$. If no distinct quadruples overlap in domain (range), then $M$ is \emph{deterministic (reversible, respectively)}. As usual, the language of a Turing machine is the set of input words on which is eventually halts.

We use the following terminology for certain classes in the arithmetical and analytical hierarchies. A set $N \subset \N$ is called $\Sigma^0_1$ if it is recursively enumerable, and $\Pi^0_1$ if its complement is. The set is called $\Sigma^1_1$, if there exists an oracle Turing machine $M$ such that
\[ N = \{ n \in \N \;|\; \exists f : \N \to \N : \mbox{$M$ never halts on input $n$ with oracle $f$} \}. \]
These are not the standard definitions of the classes, but characterizations whose proofs can be found, for example, in \cite[Theorem 1.3]{Sa90}. Hardness and completeness of a set with respect to these classes is defined using Turing reductions. When classifying subsets of other countable sets than $\N$, for example $\{0,1\}^*$, we assume that they are in some natural and computable bijection with $\N$.

In \cite{MoShGo89}, it was proved that deterministic reversible Turing machines are capable of simulating any deterministic Turing machine (first proved in \cite{Be73} for \emph{multi-tape} Turing machines). We will not go into the details of the notion of simulation, but it is easy to see that it implies the following lemmas.

\begin{lemma}
\label{lem:RevTM}
There exists a deterministic reversible Turing Machine $M$ whose language is $\Sigma^0_1$-complete.
\end{lemma}

\begin{lemma}
\label{lem:RevTMPlus}
There exists a deterministic reversible Turing Machine $M$, whose tape alphabet contains $0$, $1$ and $\#$, such that the set
\[ L = \{ w \in \{0,1\}^* \;|\; \exists u \in (0^*1)^\omega : M \mbox{~never halts on~} w \# u \} \]
is $\Sigma^1_1$-complete, and the head never steps left of the origin on right-infinite inputs.
\end{lemma}


\section{Traces and Computational Universality}

In this section, we recall the definition of computational universality for an effective symbolic system $(X,f)$, as given in \cite{DeKuBl06}. First, a \emph{clopen partition} of $X$ is a finite collection $\mathcal{C} = (C_s)_{s \in \Sigma}$ of mutually disjoint clopen subsets of $X$ such that $X = \bigcup_{s \in \Sigma} C_s$, labeled by a finite set $\Sigma$. The partition can be seen as an observation or experiment, with input $x \in X$ resulting in the unique label $\pi_{\mathcal{C}}(x) = s_0 \in \Sigma$ such that $x \in C_{s_0}$. More information can be extracted from $x$ if we apply the dynamics function $f$ to it and repeat the experiment, obtaining the result $\pi_{\mathcal{C}}(f(x)) \in \Sigma$. Iterating the idea leads to the following definition.



\begin{definition}
Let $(X,f)$ be a symbolic system, and let $\mathcal{C} = (C_s)_{s \in \Sigma}$ be a clopen partition of $X$. For $x \in X$, the \emph{$f$-itinerary of $x$ via $\mathcal{C}$} is the infinite sequence $\pi^f_{\mathcal{C}}(x) \in \Sigma^\N$ defined by $\pi^f_{\mathcal{C}}(x)_n = \pi_{\mathcal{C}}(f^n(x))$ for all $n \in \N$. The \emph{$\mathcal{C}$-trace shift of $f$} is the $\N$-shift space $\tau_{f, \mathcal{C}} = \{ \pi^f_{\mathcal{C}}(x) \;|\; x \in X \}$. In the case $X \subset S^\Z$, we denote by $\tau_{f,n}$ the trace with respect to the partition $\mathcal{C}_n = (C_w)_{w \in S^{2n+1}}$, where $C_w = \{ x \in X \;|\; x_{[-n, n]} = w \}$ for all $w \in S^{2n+1}$.
\end{definition}


\begin{example}
\label{ex:CATrace}
Let $f : S^\Z \to S^\Z$ be a cellular automaton, and let $n \in \N$. Then the trace shift $\tau_{f,n}$ is obtained by taking, for each $x \in S^\Z$, the sequence
\[ x_{[-n,n]}, f(x)_{[-n,n]}, f^2(x)_{[-n,n]}, f^3(x)_{[-n,n]}, \ldots \]
of the central words occurring in the evolution of the initial state $x$ under $f$.
\end{example}

The article \cite{DeKuBl06} defines the following two decision problems.

\begin{definition}
\label{def:Univ}
Let $(X, f)$ be an effective symbolic system. The \emph{infinite time prediction problem} asks whether we have $\tau_{f,\mathcal{C}} \cap L_A \neq \emptyset$ for a given partition $\mathcal{C} = (C_s)_{s \in \Sigma}$ and Muller automaton $A$ on the alphabet $\Sigma$. The \emph{finite time prediction problem} asks whether we have $\B(\tau_{f,\mathcal{C}}) \cap L \neq \emptyset$ for a given partition $\mathcal{C} = (C_s)_{s \in \Sigma}$ and regular language $L \subset \Sigma^*$. We say that $(X,f)$ is \emph{computationally universal} if its finite time prediction problem is $\Sigma^0_1$-complete.
\end{definition}

As hinted in Section~\ref{sec:Intro}, the most obvious way of showing that an effective symbolic system $(X, f)$ is computationally universal in the above sense is to construct a simulation of a Turing machine $M$ (or some other universal computational device) by $f$, and identify three disjoint clopen sets $C_0, C_1, C_2 \subset X$ that correspond to the classes of initial, intermediate, and halting configurations of $M$. Then the instance $0 1^* 2$ of the finite time prediction problem is positive if and only if $M$ halts on one of the the initial configurations in $C_0$.

\begin{example}
We continue Example~\ref{ex:CATrace}. In the finite time prediction problem for $(S^\Z, f)$, we are given a clopen partition of $S^\Z$, which we (for now) assume to be $\mathcal{C}_n$ for some $n \in \N$, and a regular language $L$ over the alphabet $S^{2n+1}$. For example, if $S = \{0, 1\}$ and $n = 1$, then $L$ may be given as the regular expression $([000][010])^*[111] + [100]^*[111]$ (note that the `letters' of this regular expression are binary words of length $3$). If there exists $x \in \{0, 1\}^\Z$ such that, for example, $x_{[-1, 1]} = 000$, $f(x)_{[-1,1]} = 010$ and $f^2(x)_{[-1,1]} = 111$, then the answer to the finite time prediction problem with these inputs is `yes', since the language of the $\mathcal{C}_1$-trace shift of $f$ contains the word $[000][010][111] \in L$.
\end{example}

Examples of chaotic universal effective symbolic systems and universal cellular automata were provided in \cite{DeKuBl06}, but it was explicitly left open whether a universal cellular automaton can be chaotic.

The notion of universality given in the earlier work \cite{DeKuBl05} is also equivalent to the $\Sigma^0_1$-completeness of a prediction problem, but instead of Muller automata, the definition uses a temporal logic that specifies subsets of the state space. We will not digress into this subject, as it is not necessary for stating and proving the main results of this article.

\section{Main Results}

In this section, we prove that a chaotic reversible cellular automaton can be computationally universal and even have a maximally hard infinite time prediction problem. In the proof, we use the following well-known lemma, found explicitly in \cite{Lu09}.

\begin{lemma}
\label{lem:Chaotic}
A reversible cellular automaton is chaotic (in the sense of Devaney) if and only if it is transitive.
\end{lemma}

We are now ready to state and prove our main theorem.

\begin{theorem}
\label{thm:Finite}
There exists a chaotic reversible cellular automaton whose finite-time prediction problem is $\Sigma^0_1$-complete, even when restricted to the radius-one partition $\mathcal{C}_1$. In particular, the automaton is computationally universal in the sense of Definition~\ref{def:Univ}.
\end{theorem}

\begin{proof}
We first describe the general idea of the construction. The configurations of the reversible CA we construct are divided into `compartments', each of which may contain one read-write head of a reversible Turing machine. The automaton simulates the Turing machines separately in each compartment, changing the direction of the simulation if they halt, and no information can be passed between the compartments. The compartments and the machines are also constantly shifted to the left. When one of the machines halts, it sends a signal to the right.

Now, in the trace shift we wish to see the following pattern: the left wall of a compartment, then some empty space, then the Turing machine head in its initial state followed by an input word, then another empty stretch, and finally the right-moving signal emitted by the halting machine. If we see such a pattern, the machine must have halted, since the signal cannot have come from the other side of the wall, and it cannot be the result of a left-moving signal bouncing off the wall, for no such signal was seen earlier. Finally, the transitivity of the CA follows from the constant shifting of the compartments and the fact that they never communicate, so that every central pattern of a configuration can eventually be replaced by arbitrary data.

Let $M = (\Sigma, Q, \delta, B, q_0, q_f)$ be the reversible Turing Machine of Lemma~\ref{lem:RevTM}. We may assume that the head of $M$ makes a move only on every third step, that it does not make a move in state $q_0$ for two steps when run backwards or forwards, that it always takes at least $2|w|+2$ steps for $M$ to halt on an input word $w$, and that it always halts after an even number of steps or runs forever. Denote $S = \Sigma \times (Q \cup \tilde Q \cup \{{\leftarrow}, {\rightarrow}\})$, where $\tilde Q = \{\tilde q \;|\; q \in Q\}$ is a disjoint copy of $Q$, and define a reversible cellular automaton $f_M$ on $S^\Z$ as follows.

We partition each configuration $x \in S^\Z$ into segments whose second track is of the form ${\leftarrow}^m q {\rightarrow}^n$, where $q \in Q \cup \tilde Q$, or ${\leftarrow}^{m+1} {\rightarrow}^{n+1}$, for some (possibly infinite) $m, n \geq 0$. Namely, each cell of $x$ is contained in at least one pattern of this form, and when we take the maximal ones, the partition is uniquely determined. On these segments, $f_M$ simulates a computation of $M$ (backwards in time if $q \in \tilde Q$) in a standard way. Namely, consider a cell in state $(a, q) \in \Sigma \times Q$. If $[q, a, b, r] \in \delta$ for some $b \in \Sigma$ and $r \in Q$, the cell will update to $(b, r)$. Each two-cell pattern $(a, q)(b, {\rightarrow})$ such that $[q, /, +, r] \in \delta$ for some $b \in \Sigma$ will update to $(a, {\leftarrow})(b, r)$, and analogously for a $0$- or $-$-move. In all other cases (no applicable quadruple exists, or the segment ends), the cell becomes $(a, \tilde q)$. To such cells, the quadruples and the time-reversal rule are applied in the reverse direction: $[r, b, a, q] \in \delta$ results in $(b, \tilde r)$ and so on. All cells not mentioned here retain their state. Thus the first track acts as the tape, the endpoints of the segments never move, and if the simulation cannot be carried on then it changes direction. For $x \in S^\Z$, denote by $r(x) \in S^\Z$ the configuration obtained from $x$ by changing every $q \in Q$ to $\tilde q$, and vice versa; it is easy to see that $r \circ f_M \circ r = f_M^{-1}$, which implies that $f_M$ is reversible. See Figure~\ref{fig:RevTM} for a visualization of the dynamics of $f_M$.

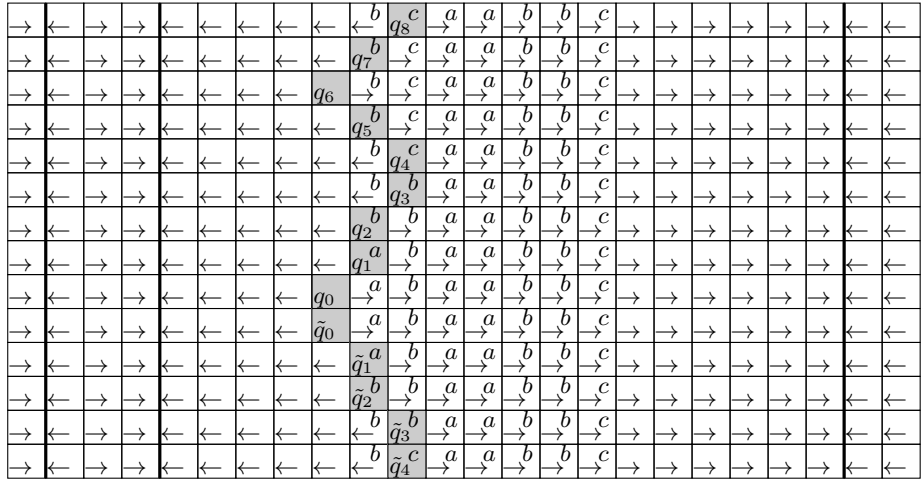
\begin{figure}[p]
\begin{center}
\begin{tikzpicture}[xscale=0.5,yscale=0.45]

\foreach \x/\y in {1/0,1/1,1/2,1/3,1/4,1/5,1/6,1/7,1/8,1/9,1/10,1/11,1/12,1/13,4/0,4/1,4/2,4/3,4/4,4/5,4/6,4/7,4/8,4/9,4/10,4/11,4/12,4/13,22/0,22/1,22/2,22/3,22/4,22/5,22/6,22/7,22/8,22/9,22/10,22/11,22/12,22/13}{
\begin{scope}[shift={(\x,\y)}]
\draw (0,0) rectangle (1,1);
\draw[very thick] (0,0) -- (0,1);
\node () at (1/3,1/4) {$\leftarrow$};

\end{scope}
}
\foreach \x/\y in {8/5}{
\begin{scope}[shift={(\x,\y)}]
\draw[fill=black!20] (0,0) rectangle (1,1);
\node () at (1/3,1/4) {$q_0$};

\end{scope}
}
\foreach \x/\y in {9/7}{
\begin{scope}[shift={(\x,\y)}]
\draw[fill=black!20] (0,0) rectangle (1,1);
\node () at (1/3,1/4) {$q_2$};
\node () at (2/3,3/4) {$b$};

\end{scope}
}
\foreach \x/\y in {9/6}{
\begin{scope}[shift={(\x,\y)}]
\draw[fill=black!20] (0,0) rectangle (1,1);
\node () at (1/3,1/4) {$q_1$};
\node () at (2/3,2/3) {$a$};

\end{scope}
}
\foreach \x/\y in {10/9}{
\begin{scope}[shift={(\x,\y)}]
\draw[fill=black!20] (0,0) rectangle (1,1);
\node () at (1/3,1/4) {$q_4$};
\node () at (2/3,2/3) {$c$};

\end{scope}
}
\foreach \x/\y in {10/8}{
\begin{scope}[shift={(\x,\y)}]
\draw[fill=black!20] (0,0) rectangle (1,1);
\node () at (1/3,1/4) {$q_3$};
\node () at (2/3,3/4) {$b$};

\end{scope}
}
\foreach \x/\y in {8/11}{
\begin{scope}[shift={(\x,\y)}]
\draw[fill=black!20] (0,0) rectangle (1,1);
\node () at (1/3,1/4) {$q_6$};

\end{scope}
}
\foreach \x/\y in {9/10}{
\begin{scope}[shift={(\x,\y)}]
\draw[fill=black!20] (0,0) rectangle (1,1);
\node () at (1/3,1/4) {$q_5$};
\node () at (2/3,3/4) {$b$};

\end{scope}
}
\foreach \x/\y in {10/13}{
\begin{scope}[shift={(\x,\y)}]
\draw[fill=black!20] (0,0) rectangle (1,1);
\node () at (1/3,1/4) {$q_8$};
\node () at (2/3,2/3) {$c$};

\end{scope}
}
\foreach \x/\y in {9/12}{
\begin{scope}[shift={(\x,\y)}]
\draw[fill=black!20] (0,0) rectangle (1,1);
\node () at (1/3,1/4) {$q_7$};
\node () at (2/3,3/4) {$b$};

\end{scope}
}
\foreach \x/\y in {5/0,5/1,5/2,5/3,5/4,5/5,5/6,5/7,5/8,5/9,5/10,5/11,5/12,5/13,6/0,6/1,6/2,6/3,6/4,6/5,6/6,6/7,6/8,6/9,6/10,6/11,6/12,6/13,7/0,7/1,7/2,7/3,7/4,7/5,7/6,7/7,7/8,7/9,7/10,7/11,7/12,7/13,8/0,8/1,8/2,8/3,8/6,8/7,8/8,8/9,8/10,8/12,8/13,23/0,23/1,23/2,23/3,23/4,23/5,23/6,23/7,23/8,23/9,23/10,23/11,23/12,23/13}{
\begin{scope}[shift={(\x,\y)}]
\draw (0,0) rectangle (1,1);
\node () at (1/3,1/4) {$\leftarrow$};

\end{scope}
}
\foreach \x/\y in {0/0,0/1,0/2,0/3,0/4,0/5,0/6,0/7,0/8,0/9,0/10,0/11,0/12,0/13,2/0,2/1,2/2,2/3,2/4,2/5,2/6,2/7,2/8,2/9,2/10,2/11,2/12,2/13,3/0,3/1,3/2,3/3,3/4,3/5,3/6,3/7,3/8,3/9,3/10,3/11,3/12,3/13,16/0,16/1,16/2,16/3,16/4,16/5,16/6,16/7,16/8,16/9,16/10,16/11,16/12,16/13,17/0,17/1,17/2,17/3,17/4,17/5,17/6,17/7,17/8,17/9,17/10,17/11,17/12,17/13,18/0,18/1,18/2,18/3,18/4,18/5,18/6,18/7,18/8,18/9,18/10,18/11,18/12,18/13,19/0,19/1,19/2,19/3,19/4,19/5,19/6,19/7,19/8,19/9,19/10,19/11,19/12,19/13,20/0,20/1,20/2,20/3,20/4,20/5,20/6,20/7,20/8,20/9,20/10,20/11,20/12,20/13,21/0,21/1,21/2,21/3,21/4,21/5,21/6,21/7,21/8,21/9,21/10,21/11,21/12,21/13}{
\begin{scope}[shift={(\x,\y)}]
\draw (0,0) rectangle (1,1);
\node () at (1/3,1/4) {$\rightarrow$};

\end{scope}
}
\foreach \x/\y in {}{
\begin{scope}[shift={(\x,\y)}]
\draw (0,0) rectangle (1,1);
\node () at (1/3,1/4) {$\leftarrow$};
\node () at (2/3,2/3) {$a$};

\end{scope}
}
\foreach \x/\y in {}{
\begin{scope}[shift={(\x,\y)}]
\draw (0,0) rectangle (1,1);
\node () at (1/3,1/4) {$\leftarrow$};
\node () at (2/3,2/3) {$c$};

\end{scope}
}
\foreach \x/\y in {9/0,9/1,9/8,9/9,9/13}{
\begin{scope}[shift={(\x,\y)}]
\draw (0,0) rectangle (1,1);
\node () at (1/3,1/4) {$\leftarrow$};
\node () at (2/3,3/4) {$b$};

\end{scope}
}
\foreach \x/\y in {9/4,9/5,11/0,11/1,11/2,11/3,11/4,11/5,11/6,11/7,11/8,11/9,11/10,11/11,11/12,11/13,12/0,12/1,12/2,12/3,12/4,12/5,12/6,12/7,12/8,12/9,12/10,12/11,12/12,12/13}{
\begin{scope}[shift={(\x,\y)}]
\draw (0,0) rectangle (1,1);
\node () at (1/3,1/4) {$\rightarrow$};
\node () at (2/3,2/3) {$a$};

\end{scope}
}
\foreach \x/\y in {10/10,10/11,10/12,15/0,15/1,15/2,15/3,15/4,15/5,15/6,15/7,15/8,15/9,15/10,15/11,15/12,15/13}{
\begin{scope}[shift={(\x,\y)}]
\draw (0,0) rectangle (1,1);
\node () at (1/3,1/4) {$\rightarrow$};
\node () at (2/3,2/3) {$c$};

\end{scope}
}
\foreach \x/\y in {9/11,10/2,10/3,10/4,10/5,10/6,10/7,13/0,13/1,13/2,13/3,13/4,13/5,13/6,13/7,13/8,13/9,13/10,13/11,13/12,13/13,14/0,14/1,14/2,14/3,14/4,14/5,14/6,14/7,14/8,14/9,14/10,14/11,14/12,14/13}{
\begin{scope}[shift={(\x,\y)}]
\draw (0,0) rectangle (1,1);
\node () at (1/3,1/4) {$\rightarrow$};
\node () at (2/3,3/4) {$b$};

\end{scope}
}
\foreach \x/\y in {9/2}{
\begin{scope}[shift={(\x,\y)}]
\draw[fill=black!20] (0,0) rectangle (1,1);
\node () at (1/3,1/3) {$\tilde q_2$};
\node () at (2/3,3/4) {$b$};

\end{scope}
}
\foreach \x/\y in {8/4}{
\begin{scope}[shift={(\x,\y)}]
\draw[fill=black!20] (0,0) rectangle (1,1);
\node () at (1/3,1/3) {$\tilde q_0$};

\end{scope}
}
\foreach \x/\y in {10/1}{
\begin{scope}[shift={(\x,\y)}]
\draw[fill=black!20] (0,0) rectangle (1,1);
\node () at (1/3,1/3) {$\tilde q_3$};
\node () at (2/3,3/4) {$b$};

\end{scope}
}
\foreach \x/\y in {}{
\begin{scope}[shift={(\x,\y)}]
\draw[fill=black!20] (0,0) rectangle (1,1);
\node () at (1/3,1/3) {$\tilde q_6$};

\end{scope}
}
\foreach \x/\y in {10/0}{
\begin{scope}[shift={(\x,\y)}]
\draw[fill=black!20] (0,0) rectangle (1,1);
\node () at (1/3,1/3) {$\tilde q_4$};
\node () at (2/3,2/3) {$c$};

\end{scope}
}
\foreach \x/\y in {9/3}{
\begin{scope}[shift={(\x,\y)}]
\draw[fill=black!20] (0,0) rectangle (1,1);
\node () at (1/3,1/3) {$\tilde q_1$};
\node () at (2/3,2/3) {$a$};

\end{scope}
}
\foreach \x/\y in {}{
\begin{scope}[shift={(\x,\y)}]
\draw[fill=black!20] (0,0) rectangle (1,1);
\node () at (1/3,1/3) {$\tilde q_5$};
\node () at (2/3,3/4) {$b$};

\end{scope}
}

\end{tikzpicture}
\end{center}
\caption{A spacetime diagram of the reversible cellular automaton $f_M$ simulating the reversible Turing machine $M$. Each row is (the central pattern of) a configuration of $S^\Z$, and time increases upwards. The shaded cells contain a Turing machine head, and the thick vertical lines mark the borders of segments. The initial state of $M$ is $q_0$, which is preceded in the simulation by the `backward' state $\tilde q_0$, and $a, b, c$ are elements of the tape alphabet. To save space, $M$ does not move only every third step in this figure.}
\label{fig:RevTM}
\end{figure}

Now, let $D = \{ \emptycell{}, \leftcell{}, \rightcell{}, \bothcell{} \}$, and denote $R = S \times D$. The cells in $D$ represent particles traveling to the left or to the right, with the fourth one containing one of each. We define a CA $g$ on $R^\Z$ that functions as follows:
\begin{enumerate}
\item Apply $f_M$ to the first track.
\item Shift each particle to its direction, unless it would cross the barrier between two segments, in which case change its direction.
\item If a cell contains the final state of $M$ in its first track, apply the bijection $\emptycell{} \leftrightarrow \rightcell{}, \leftcell{} \leftrightarrow \bothcell{}$ to the second track.
\end{enumerate}
See Figure~\ref{fig:Particles} for a visualization. Since each of the three steps is clearly reversible, so is $g$. Finally, define $h = \sigma \circ g^2$, which is likewise reversible, and denote by $\pi_S : R \to S$ and $\pi_D : R \to D$ the obvious projections from $R$.

\begin{figure}[p]
\begin{center}
\begin{tikzpicture}[xscale=0.5,yscale=0.25]

\foreach \x/\y in {5/6,8/10}{
\begin{scope}[shift={(\x,\y)}]
\draw[thick,gray] (.5,1) ellipse (.4 and .8);

\end{scope}
}
\foreach \x/\y in {0/0,0/2,0/4,0/6,0/8,0/10,0/12,2/0,2/2,2/4,2/6,2/8,2/10,2/12,4/0,4/2,4/4,4/6,4/8,4/10,4/12,5/0,5/2,5/4,5/8,5/10,5/12,6/0,6/2,6/4,6/6,6/8,6/10,6/12,8/0,8/2,8/4,8/6,8/8,8/12,9/0,9/2,9/4,9/6,9/8,9/10,9/12,10/0,10/2,10/4,10/6,10/8,10/10,10/12,11/0,11/2,11/4,11/6,11/8,11/10,11/12,12/0,12/2,12/4,12/6,12/8,12/10,12/12,13/0,13/2,13/4,13/6,13/8,13/10,13/12,14/0,14/2,14/4,14/6,14/8,14/10,14/12,16/0,16/2,16/4,16/6,16/8,16/10,16/12}{
\begin{scope}[shift={(\x,\y)}]

\end{scope}
}
\foreach \x/\y in {1/1,1/9,2/3,2/11,4/1,5/1,5/3,6/3,6/9,7/9,9/1,10/3,11/3,11/5,12/1,12/7,13/7,14/5,14/9,15/9,16/11}{
\begin{scope}[shift={(\x,\y)}]
\draw[thick,gray] (0,-1) -- (1,1);
\draw (0,-1) rectangle (1,1);

\end{scope}
}
\foreach \x/\y in {1/0,1/2,1/4,1/6,1/8,1/10,1/12,3/0,3/2,3/4,3/6,3/8,3/10,3/12,7/0,7/2,7/4,7/6,7/8,7/10,7/12,15/0,15/2,15/4,15/6,15/8,15/10,15/12}{
\begin{scope}[shift={(\x,\y)}]
\draw[very thick] (0,0) -- (0,2);

\end{scope}
}
\foreach \x/\y in {3/13,5/7,6/5,10/1,12/5,13/3,13/9,14/11}{
\begin{scope}[shift={(\x,\y)}]
\draw[thick,gray] (0,-1) -- (1,1);
\draw[thick,gray] (1,-1) -- (0,1);
\draw (0,-1) rectangle (1,1);

\end{scope}
}
\foreach \x/\y in {0/1,0/3,0/5,0/7,0/9,0/11,0/13,1/3,1/5,1/11,1/13,2/1,2/7,2/9,3/1,3/3,3/5,3/7,3/9,4/3,4/5,4/7,4/13,5/5,5/11,6/1,6/13,7/1,7/3,7/5,7/11,7/13,8/1,8/3,8/7,8/9,8/11,9/5,9/7,9/9,9/13,10/5,10/7,10/11,10/13,11/1,11/9,11/11,12/3,12/9,12/13,13/1,13/5,13/11,14/3,15/1,15/3,15/5,15/11,15/13,16/1,16/3,16/7,16/9,16/13}{
\begin{scope}[shift={(\x,\y)}]
\draw (0,-1) rectangle (1,1);

\end{scope}
}
\foreach \x/\y in {1/7,2/5,2/13,3/11,4/9,4/11,5/9,5/13,6/7,6/11,7/7,8/5,8/13,9/3,9/11,10/9,11/7,11/13,12/11,13/13,14/1,14/7,14/13,15/7,16/5}{
\begin{scope}[shift={(\x,\y)}]
\draw[thick,gray] (1,-1) -- (0,1);
\draw (0,-1) rectangle (1,1);

\end{scope}
}

\end{tikzpicture}
\end{center}
\caption{A spacetime diagram of $g$, showing the dynamics of the particles. The thick lines mark the borders of segments, and the circles denote the final state of $M$. Other information from the first track is not shown. The particles are drawn in gray for clarity.}
\label{fig:Particles}
\end{figure}
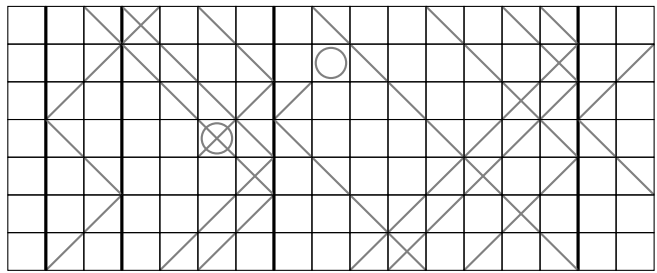

Now, let $w \in \Sigma^*$ be arbitrary, and denote by $L(w) \subset R^*$ the regular language
\begin{equation}
\label{eq:Regu}
	(B,\rightarrow,\emptycell{})
	(B,\leftarrow,\emptycell{})^*
	(B,q_0,\emptycell{})
	(w \times (\rightarrow,\emptycell{})^{|w|})
	(B,\rightarrow,\emptycell{})^*
	(B,\rightarrow,\rightcell{})
\end{equation}
The words of the language $L(w)$ consist of the left border of a segment, then some `empty' cells, followed by the initial state of $M$ and its input word, then more empty cells, and finally a right-moving particle. Let $U(w) \subset (R^3)^\N$ be the open set of all configurations that have a prefix $v \in (R^3)^*$ such that the middle components of the triples in $v$ form a word in $L(w)$, and none of the letters of $v$ contain a left-moving particle. We claim that $\tau_{h,1} \cap U(w)$ is nonempty if and only if $M$ halts on $w$, which is $\Sigma^0_1$-complete.

First, assume that $M$ halts on $w$ after exactly $2n+2$ steps, where $n \geq |w|$, and define $x \in R^\Z$ by
\[ x_i = \left\{ \begin{array}{ll}
	(B,\rightarrow,\emptycell{}), & \mbox{if~} i \leq -n, \\
	(B,\leftarrow,\emptycell{}), & \mbox{if~} -n < i < 0, \\
	(B,q_0,\emptycell{}), & \mbox{if~} i = 0, \\
	(w_{i-1},\rightarrow,\emptycell{}), & \mbox{if~} 1 \leq i \leq |w|, \\
	(B,\rightarrow,\emptycell{}), & \mbox{if~} i > |w|. \\
\end{array} \right.\]
Denote $y = h^{-n}(x)$. We claim that $(h^i(y)_{[-1,1]})_{i \in \N} \in U(w)$. To prove that, we first remark that for all $k, \ell \in \Z$ such that $3|\ell| \geq |k|$ we have $\pi_S(g^k(x)_\ell) = \pi_S(x_\ell)$, since only the single Turing Machine head can introduce changes to the $S$-component of the configuration, and it only moves every third step by assumption. This means that for all $i \in \Z$ we have $\pi_S(h^i(y)_0) = \pi_S(x_{i-n})$.

Since $M$ does not halt in $2n$ steps, the second track of every configuration in $\{ h^i(y) \;|\; i \in \{0, \ldots, 2n\} \}$ contains no particles. But since $M$ halts at step $2n+2$, there exists $k \in \Z$ with $|k| \leq \frac{n}{3}$ such that $\pi_D(h^{2n+1}(y)_{k-n-1}) = \pi_D(g^{2n+2}(x)_k) = \rightcell{}$. Since the single Turing Machine head will not enter the final state for another $4n+4$ steps, we have $\pi_D(h^i(y)_0) = \emptycell{}$ for every $i \in \{2n+1, \ldots, 3n-k+1\}$, and $\pi_D(h^{3n-k+2}(y)_0) = \rightcell{}$. Furthermore, no letter in $(h^i(y)_{[-1,1]})_{i \in \N}$ contains a left-moving particle, and together with the previous paragraph, this shows that $(h^i(y)_{[-1,1]})_{i \in \N} \in U(w)$, and thus $\tau_{h,1} \cap U(w) \neq \emptyset$. See Figure~\ref{fig:Dynamics} for a visualization that will also be helpful in the converse direction.

\begin{figure}[htp]
\begin{center}
\begin{tikzpicture}[scale=1.85]

\fill[black!15] (1,2) -- (1.1,2.3) -- (1.1,2.9) -- (1.2,3.2) -- (1.2,3.8) -- (1.3,4.1) -- (1.3,2);

\node (w) at (1.5,1.6) {$w$};
\node (w2) at (1.16,2) {};
\draw (w) edge [out=180,in=270,->] (w2);

\draw[very thick] (0,0) -- (0,7.5);
\draw (-.2,2) -- (4.1,2);
\node[left] () at (-.2,2) {$x$};
\draw[dotted] (0,0) -- (3.75,7.5);
\fill (1,2) circle (.05);
\draw[densely dotted] (1,2)
-- ++(.1,.3) -- ++(-.1,.3) -- ++(.1,.3) -- ++(.1,.3) -- ++(-.1,.3) -- ++(.1,.3) -- ++(.1,.3) -- ++(.1,.3) -- ++(-.1,.3);
\draw[densely dotted] (1,2)
-- ++(.1,-.3) -- ++(-.1,-.3) -- ++(.1,-.3) -- ++(.1,-.3) -- ++(-.1,-.3) -- ++(.1,-.3) -- ++(2/30,-.2);
\draw[densely dotted] (1.3,4.7)
-- ++(.1,.3) -- ++(-.1,.3) -- ++(-.1,.3) -- ++(-.1,.3) -- ++(.1,.3) -- ++(-.1,.3) -- ++(-.1,.3) -- ++(.1,.3) -- ++(-.1,.3) -- ++(1/30,.1);
\fill (1,7.4) circle (.05);
\draw (1.3,4.7) -- (4.1,7.5);
\draw[dashed] (1.3,4.7) -- (0,3.4) -- (3.4,0);
\draw (-.2,0) -- (4.1,0);
\node[left] () at (-.2,0) {$g^{-2n}(x) = \sigma^n(y)$};
\draw (-.2,4.7) -- (4.1,4.7);
\draw (1.3,4.7) -- ++(0,.1);
\node[above] () at (1.3,4.8) {$k$};
\filldraw[fill=black!30] (1.3,4.7) circle (.05);
\node[left] () at (-.2,4.7) {$g^{2n+2}(x)$};
\node[below] () at (0,0) {$-n$};
\draw (1,2) -- ++(0,.1);
\node[above] () at (1,2.1) {$0$};
\draw (3.4,0) -- ++(0,.1);
\node[below] () at (3.4,0) {$m$};
\draw (3.4,6.8) -- ++(0,.1);
\node[above] () at (3.4,6.9) {$m$};
\draw (-.2,6.8) -- (4.1,6.8);
\node[left] () at (-.2,6.8) {$g^{2m}(x)$};

\draw (1.3,1.95) -- ++(0,.1);


\end{tikzpicture}
\end{center}
\caption{A schematic spacetime diagram of $g$. The black (gray) circles represent the initial (final) states of $M$ (at coordinates $0$ and $k$, respectively), and the densely dotted line traces the read-write head of $M$ as it carries out its computation. The sparsely dotted line corresponds to the central coordinates of the configurations $h^{i-n}(x)$ for $i \in \N$. The thick vertical line is a border of segments, the diagonal line is a particle, and the dashed continuation is its hypothetical path assuming that $M$ never halted. The particle meets the sparsely dotted line at coordinate $m$. Each horizontal line represents a configuration. The shaded area represents the letters of the input word $w \in \Sigma^*$ in $x$ that the head of $M$ has not yet read (which the sparsely dotted line goes through).}
\label{fig:Dynamics}
\end{figure}
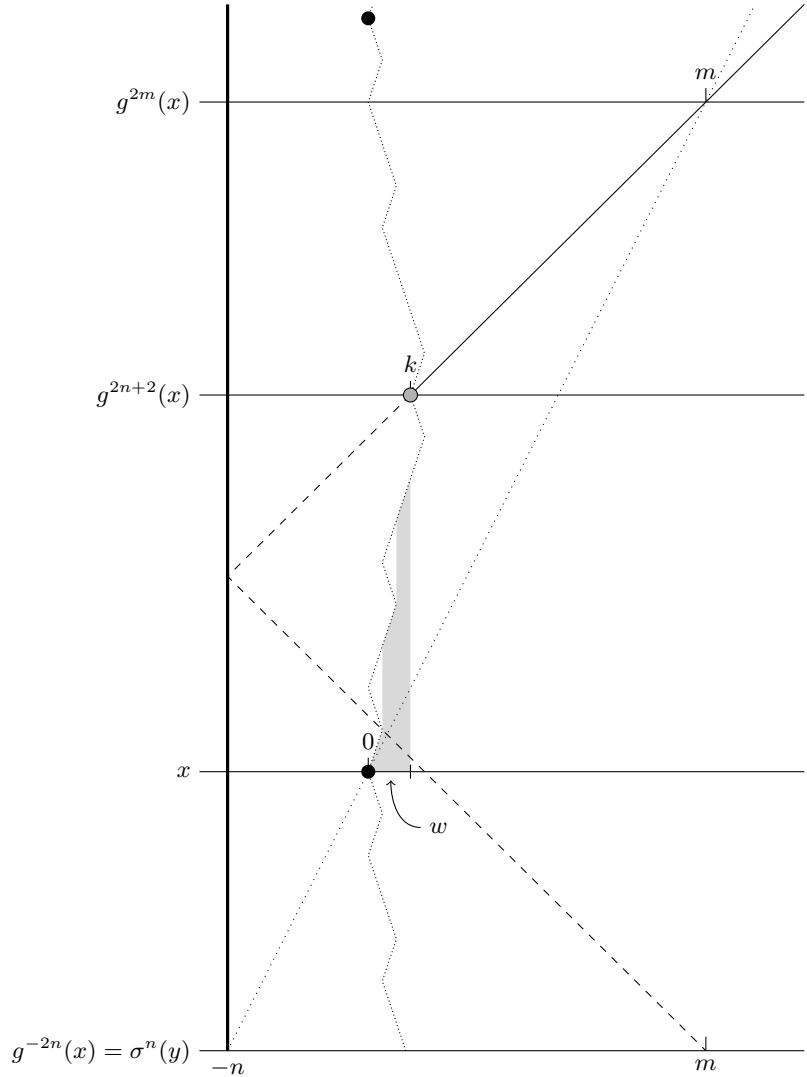

Second, assume that $y \in R^\Z$ is such that $(h^i(y))_{i \in \N} \in U(w)$, with the $*$-symbols in the definition of $L(w)$ being replaced by numbers $n-1, m-|w|-1 \in \N$, so that we have $h^n(y)_0 = (B,q_0,\emptycell{})$ and $h^{n+m}(y)_0 = (B,\rightarrow,\rightcell{})$. Let $x = h^n(y)$. Since the segments used in the simulation of $M$ never move under the action of $g$, the interval $I = [-n,m]$ is contained in a single segment of $y$, and contains its left endpoint. The segment contains a Turing Machine head, and as above, for all $k, \ell \in \Z$ such that $3|\ell| \geq 2|k|$, we have $\pi_S(h^{k+n}(y)_{\ell-k}) = \pi_S(g^{2k}(x)_\ell) = \pi_S(x_\ell)$. In particular, the word $\pi_S(x_I)$ contains a Turing Machine head in state $q_0$ followed by the input word $w$, surrounded by the blank symbols.

Consider now the two remaining tracks of $x$ and $y$. We know from the above that $g^{2m+2n}(y)_{m+n} = h^{m+n}(y)_0 = (B,\rightarrow,\rightcell{})$ (the intersection of the particle and the densely dotted line in Figure~\ref{fig:Dynamics}). This implies that either the coordinate $g^{2m+2n-i}(y)_{m+n-i}$ contains the final state of $M$ for some $i \in \{0, \ldots, m+n\}$, or each of them contains a right-moving particle. In the latter case, $g^{m+n-1}(y)_0$ contains a left-moving particle, since it is next to a segment border, and so does $g^{m+n-1-j}(y)_j$ for all $j \in \{0, \ldots, m+n\}$ (see the dashed line in the figure). Now, at $j = 0$ we have $m+n-1-j > 2j$, while at $j = m+n$, the opposite holds. Thus we have $|(m+n-1-j_0)/2 - j_0| \leq 1$ for some $j_0$ such that $m+n-1-j_0$ is even; denote $m+n-1-j_0 = 2\ell$. Then $0 \leq \ell < m+n$, and the cell
\[ h^\ell(y)_b = g^{2\ell}(y)_{\ell+b} = g^{m+n-1-j_0}(y)_{j_0} \]
contains a left-moving particle for some $b \in \{-1,0,1\}$ (the intersection of the dashed line with the densely dotted line in the figure). But this is impossible since $(h^i(y))_{i \in \N} \in U(w)$, so the choice that none of the cells $g^{2m+2n-i}(y)_{m+n-i} = g^{2m-i}(x)_{m-i}$ contain the final state of $M$ was incorrect. Thus one of them does, implying that $M$ eventually halts, since the initial and final states lie in the same segment (at different times). This finishes the proof of $\tau_{h,1} \cap U(w) \neq \emptyset$ being equivalent to $M$ halting on $w$.

Finally, we show that $h$ is chaotic, and by Lemma~\ref{lem:Chaotic}, it suffices to prove transitivity. The proof is standard for reversible CA that have `shifting barriers', in our case borders of segments. Let thus $u, v \in R^n$ be two words of the same length $n$. Define $w = (B,\leftarrow,\emptycell{}) (B,\rightarrow,\emptycell{})$. Then for all $x, y \in R^\Z$ with $x_{[0,1]} = y_{[0,1]} = x_{[n+2,n+3]} = y_{[n+2,n+3]} = w$ and $x_{[2,n+1]} = y_{[2,n+1]}$, we have $g^i(x)_{[2,n+1]} = g^i(y)_{[2,n+1]}$ for all $i \in \Z$. This is because the evolution of a segment under $g$ is independent of other segments.

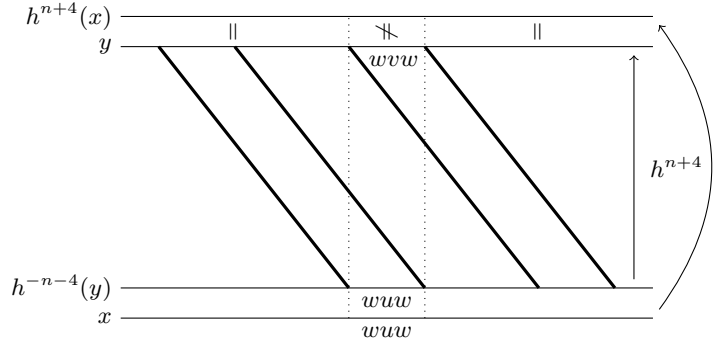
\begin{figure}[htp]
\begin{center}
\begin{tikzpicture}[yscale=.8]

\draw (0,0) -- (7,0);
\node[left] () at (0,0) {$h^{-n-4}(y)$};

\draw (0,-.5) -- (7,-.5);
\node[left] () at (0,-.5) {$x$};

\draw (0,4.5) -- (7,4.5);
\node[left] () at (0,4.5) {$h^{n+4}(x)$};
\node () at (1.5,4.25) {\rotatebox{90}{$=$}};
\node () at (3.5,4.25) {\rotatebox{90}{$\neq$}};
\node () at (5.5,4.25) {\rotatebox{90}{$=$}};

\draw (0,4) -- (7,4);
\node[left] () at (0,4) {$y$};

\node (xd) at (7,-.5) {};
\node (yd) at (6.75,0) {};
\node (yu) at (6.75,4) {};
\node (xu) at (7,4.5) {};
\draw[->] (xd) edge [bend right] node [left] {$h^{n+4}$} (xu);
\draw[->] (yd) edge (yu);

\node[below] () at (3.5,-.5) {$wuw$};
\node[below] () at (3.5,0) {$wuw$};
\node[below] () at (3.59,4) {$wvw$};

\draw[very thick] (.5,4) -- (3,0);
\draw[very thick] (1.5,4) -- (4,0);
\draw[very thick] (3,4) -- (5.5,0);
\draw[very thick] (4,4) -- (6.5,0);

\draw[dotted] (3,-.5) -- (3,4.5);
\draw[dotted] (4,-.5) -- (4,4.5);

\end{tikzpicture}
\end{center}
\caption{A schematic spacetime diagram of $h$. The vertical lines represent configurations, and each thick line is a border of segments. The dotted lines mark the interval $[0, n+3]$.}
\label{fig:Chaotic}
\end{figure}

Now, let $x \in R^\Z$ be such that $x_{[0,n+3]} = w u w$, and let $y \in R^\Z$ be defined by
\[ y_i = \left\{ \begin{array}{ll}
	(w v w)_i, & \mbox{if~} i \in [0,n+3], \\
	h^{n+4}(x)_i, & \mbox{otherwise.}
\end{array} \right. \]
By applying the above argument to $y$, which satisfies $y_{[n+4,n+5]} = y_{[2n+6,2n+7]} = w$, we have that $h^{-n-4}(y)_{[2,n+1]} = u$, and by definition $y_{[2,n+1]} = v$. This shows that $h$ is transitive, and thus chaotic. See Figure~\ref{fig:Chaotic} for a visualization of this argument.
\qed
\end{proof}

The regular expression~\eqref{eq:Regu} used in this construction is \emph{local} in the sense that it can be recognized by a DFA whose state only depends on the $n$ symbols it last read. This also applies to the regular language used in the definition of $U(w)$, where the small extra condition of not having left-moving particles in the neighboring coordinates was added.

The reversible cellular automaton we constructed above also has a maximally hard infinite time prediction problem, provided that we choose the machine $M$ correctly.

\begin{theorem}
\label{thm:Infinite}
There exists a chaotic reversible cellular automaton whose infinite-time prediction problem is $\Sigma^1_1$-complete, even when restricted to the radius-one partition $\mathcal{C}_1$.
\end{theorem}

\section{Further Discussion}


The automaton $h$ constructed in Theorem~\ref{thm:Finite} and Theorem~\ref{thm:Infinite} is chaotic, but it is not \emph{expansive}. A reversible cellular automaton $h : S^\Z \to S^\Z$ is expansive, if there exists $r \in \N$ such that for all distinct pairs $x \neq y \in S^\Z$, there exists $k \in \Z$ with $h^k(x)_{[-r,r]} \neq h^k(y)_{[-r,r]}$. Intuitively, this means that all discrepancies between two configurations are propagated to the left and to the right, if we consider both their past and future itineraries. Expansivity can thus be seen as an extreme form of sensitivity to initial conditions, and it makes sense to ask whether an expansive CA can be computationally universal.

Unfortunately, the trace shifts of expansive cellular automata are a deep and mysterious subject, and not much is known about them. In \cite{Bo04}, it was shown that if $h$ is an expansive CA, then all of its wide enough traces ($\tau_{h,n}$ for large enough $n \in \N$) have a property called \emph{total chain transitivity}, which makes it difficult to find much structure in them. Also, in \cite{Na08}, it was shown that if $h$ has memory or anticipation $0$, then all wide enough traces are actually SFTs. It is currently unknown whether this holds for all expansive CA (as conjectured in \cite{Na95}), but it would directly imply that their prediction problems are decidable, at least for a fixed clopen partition.



\section{Acknowledgments}

The authors are thankful to the anonymous referees for their valuable comments that helped to improve the quality and readability of this article.

\bibliographystyle{plain}
\bibliography{../../../bib/bib}{}

\def\ocirc#1{\ifmmode\setbox0=\hbox{$#1$}\dimen0=\ht0 \advance\dimen0
  by1pt\rlap{\hbox to\wd0{\hss\raise\dimen0
  \hbox{\hskip.2em$\scriptscriptstyle\circ$}\hss}}#1\else {\accent"17 #1}\fi}
\begin{thebibliography}{10}

\bibitem{AxGl11}
Holger~Bock Axelsen and Robert Gl{\"u}ck.
\newblock What do reversible programs compute?
\newblock In {\em Foundations of software science and computational
  structures}, volume 6604 of {\em Lecture Notes in Comput. Sci.}, pages
  42--56. Springer, Heidelberg, 2011.

\bibitem{Be73}
C.~H. Bennett.
\newblock Logical reversibility of computation.
\newblock {\em IBM J. Res. Develop.}, 17:525--532, 1973.

\bibitem{Bo04}
Mike Boyle.
\newblock Some sofic shifts cannot commute with nonwandering shifts of finite
  type.
\newblock {\em Illinois J. Math.}, 48(4):1267--1277, 2004.

\bibitem{Co04}
Matthew Cook.
\newblock Universality in elementary cellular automata.
\newblock {\em Complex Systems}, 15(1):1--40, 2004.

\bibitem{DeKuBl06}
Jean-Charles Delvenne, Petr K{\r{u}}rka, and Vincent Blondel.
\newblock Decidability and universality in symbolic dynamical systems.
\newblock {\em Fund. Inform.}, 74(4):463--490, 2006.

\bibitem{DeKuBl05}
Jean-Charles Delvenne, Petr K{\ocirc{u}}rka, and Vincent~D. Blondel.
\newblock Computational universality in symbolic dynamical systems.
\newblock In {\em Machines, computations, and universality}, volume 3354 of
  {\em Lecture Notes in Comput. Sci.}, pages 104--115. Springer, Berlin, 2005.

\bibitem{De89}
Robert~L. Devaney.
\newblock {\em An introduction to chaotic dynamical systems}.
\newblock Addison-Wesley Studies in Nonlinearity. Addison-Wesley Publishing
  Company Advanced Book Program, Redwood City, CA, second edition, 1989.

\bibitem{DuRo99}
Bruno Durand and Zsuzsanna R{\'o}ka.
\newblock The game of life: universality revisited.
\newblock In {\em Cellular automata ({S}aissac, 1996)}, volume 460 of {\em
  Math. Appl.}, pages 51--74. Kluwer Acad. Publ., Dordrecht, 1999.

\bibitem{Ku97}
Petr K{\r{u}}rka.
\newblock On topological dynamics of {T}uring machines.
\newblock {\em Theoret. Comput. Sci.}, 174(1-2):203--216, 1997.

\bibitem{Lu09}
Ville Lukkarila.
\newblock Sensitivity and topological mixing are undecidable for reversible
  one-dimensional cellular automata.
\newblock Technical Report 927, TUCS, 2009.

\bibitem{MoHa89}
Kenichi Morita and Masateru Harao.
\newblock Computation universality of one-dimensional reversible (injective)
  cellular automata.
\newblock {\em The Transactions of The IEICE}, E72-E(6):758--762, 1989.

\bibitem{MoShGo89}
Kenichi Morita, Akihiko Shirasaki, and Yoshifumi Gono.
\newblock A 1-tape 2-symbol reversible turing machine.
\newblock {\em The Transactions of the IEICE}, E72:223--228, 1989.

\bibitem{Na95}
Masakazu Nasu.
\newblock Textile systems for endomorphisms and automorphisms of the shift.
\newblock {\em Mem. Amer. Math. Soc.}, 114(546):viii+215, 1995.

\bibitem{Na08}
Masakazu Nasu.
\newblock Textile systems and one-sided resolving automorphisms and
  endomorphisms of the shift.
\newblock {\em Ergodic Theory and Dynamical Systems}, 28:167--209, 2 2008.

\bibitem{Sa90}
G.E. Sacks.
\newblock {\em Higher recursion theory}.
\newblock Perspectives in mathematical logic. Springer-Verlag, 1990.

\bibitem{Sm71}
Alvy~Ray Smith, III.
\newblock Simple computation-universal cellular spaces.
\newblock {\em J. Assoc. Comput. Mach.}, 18:339--353, 1971.

\bibitem{Th90}
Wolfgang Thomas.
\newblock Automata on infinite objects.
\newblock In {\em Handbook of theoretical computer science, {V}ol.\ {B}}, pages
  133--191. Elsevier, Amsterdam, 1990.

\bibitem{Wo84}
Stephen Wolfram.
\newblock Universality and complexity in cellular automata.
\newblock {\em Phys. D}, 10(1-2):1--35, 1984.
\newblock Cellular automata (Los Alamos, N.M., 1983).

\end{thebibliography}

\vfill
\pagebreak

\section*{Appendix}

Proof of Theorem~\ref{thm:Infinite}:

\begin{proof}
The CA $h : R^\Z \to R^\Z$ and the proof idea are exactly the same as in Theorem~\ref{thm:Finite}, but the machine $M$ now needs to satisfy the claim of Lemma~\ref{lem:RevTMPlus}.

Now, let $w \in \{0,1\}^*$, and define the Muller-recognizable language $L(w) \subset R^\N$ by the infinite regular expression
\begin{equation}
\label{eq:InfRegu}
	(B,\rightarrow,\emptycell{})
	(B,q_0,\emptycell{})
	(w \times (\rightarrow,\emptycell{})^{|w|})
	(\#,\rightarrow,\emptycell{})
	((0,\rightarrow,\emptycell{})^* (1,\rightarrow,\emptycell{}))^\omega
\end{equation}
It is similar to \eqref{eq:Regu}, except that the head of $M$ is situated right next to the end of the segment, and after the input word $w$ we may have an infinite tail of $0$s and $1$s, but no particles at all. As before, let $U(w) \subset (R^3)^\N$ be the set of configurations whose middle letters form a configuration in $L(w)$, and none of the letters of which contains a left-moving particle.

With a proof mimicking that of Theorem~\ref{thm:Finite}, we can now show that $\tau_{f,1} \cap U(w) \neq \emptyset$ if and only if there exists $u \in \{0,1\}^\N$ such that $M$ never halts on $w \# u$, and this is $\Sigma^1_1$-complete by the choice of $M$. Namely, if $M$ never halts, then the preimage of a configuration $x \in R^\Z$ containing the initial configuration of $M$ with input $w \# u$ for some $u \in (0^*1)^\omega$ has its trace in $U(w)$ since no particles are ever introduced. Conversely, if such a configuration exists, then it necessarily simulates a non-halting computation of $M$, since a particle must be either created or destroyed at the time of halting, both of which are impossible. Furthermore, $h$ is chaotic by the same argument as before.
\qed
\end{proof}

\end{document}